\documentclass[11pt,oneside,reqno]{amsart}

\usepackage[utf8]{inputenc}
\usepackage[T1]{fontenc}
\usepackage[russian,ngerman,english]{babel}
\usepackage{lmodern}

\usepackage{hyperref}

\usepackage{amsmath}
\usepackage{amsthm}
\usepackage{amsfonts}
\usepackage{amssymb}
\usepackage{amscd}
\usepackage{amsbsy}

\usepackage{pdfpages}
\usepackage{fancyhdr}
\usepackage{geometry}
\usepackage{setspace}
\usepackage{xcolor}
\usepackage{pifont}

\usepackage{graphicx}
\usepackage{tabularx}
\usepackage{epsfig}
\usepackage[all]{xy}
\usepackage{tikz}
\usetikzlibrary{matrix}

\usepackage{fixmath}

\usepackage{appendix}

\usepackage{enumerate}

\newcommand{\lqs}{\leqslant}
\newcommand{\gqs}{\geqslant}

\newcommand{\lb}{\label}

\newcommand{\ol}{\overline}
\newcommand{\dom}{\operatorname{dom}}

\newtheorem{theorem}{Theorem}[section]
\newtheorem{lemma}[theorem]{Lemma}

\newtheorem{corollary}[theorem]{Corollary}

\theoremstyle{definition}
\newtheorem{definition}[theorem]{Definition}

\newtheorem*{remark}{Remark}

\newtheorem{example}[theorem]{Example}%%%%%%%%%%%%%%%%%%%%%%%%%%%%%%%%%%%%%%%%%%%%%%%%%%%%%%%%%%%

\numberwithin{equation}{section}

\newcommand{\supp}{{\operatorname{supp}}}

\newcommand{\ac}{\mathrm{ac}}

\newcommand{\cB}{{\mathcal{B}}}

\newcommand{\cD}{{\mathcal{D}}}

\newcommand{\cF}{{\mathcal{F}}}

\newcommand{\cH}{{\mathcal{H}}}

\newcommand{\dd}{\mathrm{d}}

\renewcommand{\Re}{\operatorname{Re}}
\renewcommand{\Im}{\operatorname{Im}}

\newcommand{\ess}{\text{\rm{ess}}}

\newcommand{\bbZ}{\mathbb{Z}}

\newcommand{\bbR}{\mathbb{R}}
\newcommand{\bbC}{\mathbb{C}}

\newcommand{\AC}{\mathrm{AC}}
\newcommand{\loc}{\mathrm{loc}}
\newcommand{\unif}{\mathrm{unif}}

\DeclareMathOperator{\intt}{int}

\begin{document} 
	
\numberwithin{equation}{section}
\allowdisplaybreaks

\title[Modified Jost solutions]{Modified Jost solutions of Schr\"odinger operators with locally $H^{-1}$ potentials}
%OR
%\title{WKB behavior of eigenfunctions of Schr\"odinger operators with locally $H^{-1}$ potentials}

%
%
\author[M.\ Luki{\'c}]{Milivoje Luki{\'c}}
\address{ Department of Mathematics, Rice University, Houston, TX 77005, USA}
\email{milivoje.lukic@rice.edu}
\thanks{M.L.\ was supported in part by NSF grant DMS--2154563.}

\author[X.\ Wang]{Xingya Wang}
\address{ Department of Mathematics, Rice University, Houston, TX 77005, USA}
\email{xw62@rice.edu}
\thanks{X.W.\ was supported in part by NSF grant DMS--2154563.}

%\dedicatory{}

\date{\today}
%\subjclass[]{}

%%%%%%%%%%%%
%%%%%%%%%%%%
\begin{abstract}
We study Jost solutions of Schr\"odinger operators with potentials which decay with respect to a local $H^{-1}$ Sobolev norm; in particular, we generalize to this setting the results of Christ--Kiselev for potentials between the integrable and square-integrable rates of decay, proving existence of solutions with WKB asymptotic behavior on a large set of positive energies. This applies to new classes of potentials which are not locally integrable, or have better decay properties with respect to the $H^{-1}$ norm due to rapid oscillations.
\end{abstract}

\maketitle

%\vspace*{-3mm}
{\scriptsize{\tableofcontents}}
%\normalsize

%%%%%%%%%%%%%%%%%%%%%%%%%%%%%%
%%%%%%%%%%%%%%%%%%%%%%%%%%%%%%

\section{Introduction}

One-dimensional Schr\"odinger operators $- \frac{d^2}{dx^2} +V$ with decaying potentials $V$ are often studied by  comparison with the free Schr\"odinger operator (case $V=0$).  In particular, Jost solutions are eigensolutions which asymptotically behave similarly to eigensolutions of the free Schr\"odinger operator. More precisely, let $E > 0$ and
\[
k = \sqrt E.
\]
If $V$ is compactly supported,  there are solutions of $-u''+ Vu = Eu$ such that $u(x) = e^{i kx}$ for all large enough $x$, and if $V \in L^1((0,\infty))$, there are solutions of $-u'' + Vu = Eu$ with asymptotic behavior
\[
u(x) = e^{i k x} +o(1), \qquad x \to + \infty.
\]
Note that the $L^1$ condition serves two purposes: local integrability of $V$ is a common assumption on the potential, affecting everything from self-adjointness on \cite{TeschlBook,LukicBook}; moreover, the global $L^1$ condition serves as a fast decay assumption. To allow slower decay at infinity without imposing stronger local assumptions, one often uses the spaces
\[
\ell^p(L^q) = \{ f \mid \sum_{n=0}^\infty \lVert f \chi_{(n,n+1)} \rVert_q^p < \infty \}.
\]
For weaker  decay assumptions on $V$, existence of modified Jost solutions with the WKB asymptotic behavior
\begin{equation}\label{WKBL1}
u(x,E) = e^{ik x - \frac{i}{2k} \int_0^x V(t)\,dt} + o(1), \qquad x \to +\infty
\end{equation}
was studied by Kiselev \cite{Kiselev96,Kiselev98}, Remling \cite{Remling98,RemlingPAMS1998,RemlingDuke2000}, Christ--Kiselev \cite{ChristKiselev98,CAK2001CMP,ChristKiselev,ChristKiselev02}; for instance, if $V \in \ell^p(L^1)$ for some $p \in (1,2)$, eigensolutions obeying \eqref{WKBL1} exist for Lebesgue-a.e.\ $E > 0$ \cite{ChristKiselev}, and with some power law decay in the form of a condition $(1+x)^\gamma V \in \ell^p(L^1)$ with $\gamma>0$, $p\in (1,2]$, there is a bound on the Hausdorff measure of the bad set of positive energies without the WKB asymptotic behavior \cite{CAK2001CMP}.

Jost solutions are bounded, so through subordinacy theory \cite{MR915965,Stolz92,MR1738043,MR1761504,EichingerLukicSimanek}, they imply absolute continuity of the spectral measure on the corresponding set of energies. They are used in one-dimensional scattering theory to show existence and completeness of wave operators \cite{Agmon75, Hormander76,RS3,Kato95,ChristKiselev02}, see also \cite{Denisov04,Bessonov20,DenisovMohamed21,BessonovDenisov23}.  They are also the basis for inverse scattering on the line \cite{DeiftTrubowitz}, and they are related to Szeg{\H{o}} asymptotics \cite{DamanikSimon06,DamanikSimon06b}. The generalized Jost solutions in this paper can serve as the basis for further investigations in all these directions.

In this paper, we study Schr\"odinger operators with locally $H^{-1}$ potentials, which is more general than the local $L^1$ assumption mentioned above. Their study was initiated by Hryniv–Mykytyuk \cite{HrMyk01,MR2978191} in the full-line setting, within a long literature on operators with singular coefficients including \cite{WeidmannODEs,MR1756602,EGNT13}. In particular, \cite{HrMyk01} used an explicit molifier $\phi \in H^1(\bbR)$ with $\supp \phi = [-1,1]$ such that $\sum_{n\in\bbZ} \phi( \cdot - n) = 1$; a potential $V$ is locally $H^{-1}$ if $V \phi(\cdot - n) \in H^{-1}(\bbR)$ for all $n$. They constructed a decomposition
\begin{equation}\label{HMdecomposition}
V = \sigma' + \tau
\end{equation}
with $\sigma\in L^2_\loc(\bbR)$, $\tau\in L^1_\loc(\bbR)$. This decomposition is local, in the sense that values of $\sigma,\tau$ on $(a,b)$ only depend on the action of the distribution $V$ on test functions $\phi \in C_0^\infty((a-c,b+c))$ for some universal constant $c>0$, and obeys 
\begin{equation}\label{HMinequality}
C^{-1} \sup_{n} \lVert V \phi(\cdot -n) \rVert_{H^{-1}(\bbR)} \le \sup_x \left( \lVert \sigma \chi_{(x,x+1)} \rVert_2 +  \lVert \tau \chi_{(x,x+1)} \rVert_1 \right) \le C \sup_{n} \lVert V \phi(\cdot -n) \rVert_{H^{-1}(\bbR)}
\end{equation}
for some universal constant $C$. If these quantities are finite, $V$ is said to be locally uniformly $H^{-1}$. For particular choices of $V$, one does not have to use the exact  $\sigma, \tau$ constructed by \cite{HrMyk01}, and relevant statements are independent of the choice of decomposition.

This level of generality obviously allows a greater family of potentials to be studied, including Dirac delta terms $\delta_{x_0}$ and Coulomb singularities $\lvert x - x_0 \rvert^{-1}$ at internal points $x_0$. There are also other motivations for the locally $H^{-1}$ setting. The decomposition \eqref{HMdecomposition} is related to the Miura transformation and the Riccati representation \cite{Korotyaev,KappelerTopalov} for periodic $V$; however, the non-periodic, infinite interval setting requires two functions $\sigma, \tau$, where $\tau$ takes the role of a local average, and $\sigma'$ contains the less smooth part of the potential. The representation \eqref{HMdecomposition} was observed to be useful even for $V \in L^1_\loc$ \cite{DenisovMohamed21}, and can be motivated also through the connection to a square of a Dirac operator.

The final motivation is that the $H^{-1}$ norm is less sensitive to rapid oscillations; thus, rapidly oscillating potentials can seem decaying with respect to a local $H^{-1}$ norm, even if they are not classically decaying, or they can seem to be decaying at a faster rate. We will illustrate this below with Example~\ref{WKBoscillatoryexample}.

The half-line setting is natural for the goals of this paper, so we consider half-line distributions $V \in \cD'(\bbR_+)$ of the form \eqref{HMdecomposition} for some $\sigma \in L^2_{\loc,\unif}(\bbR_+)$, $\tau \in L^1_{\loc,\unif}(\bbR_+)$. If formally necessary, one can assume $\sigma(x)=\tau(x) = 0$ for $x < 0$. The quasiderivative \cite{WeidmannODEs,MR1756602} of a locally absolutely continuous function $u$ is
 \[
 u^{[1]}:=u'-\sigma u,
 \]
and the formal action of the Schr\"odinger operator is defined on the local domain
\begin{equation}
\mathfrak D:=\{u\in \AC_\loc(\bbR_+): u^{[1]}\in \AC_\loc(\bbR_+)\}
\end{equation}
by
\begin{equation}
\ell u:=-(u^{[1]})' - \sigma u^{[1]} + (\tau-\sigma^2)u. \lb{difex}
\end{equation}
Half-line self-adjoint Schr\"odinger operators $H$ on the Hilbert space $L^2(\bbR_+)$ are obtained \cite{HrMyk01,LSW} by restricting $\ell$ to the domains
\[
\dom(H):=\{u\in L^2(\bbR_+) \mid u \in \mathfrak D,\, \ell u\in L^2(\bbR_+), \, u(0)\cos(\alpha) +u^{[1]}(0)\sin(\alpha)=0 \}
\]
where $\alpha$ labels the boundary condition at $0$. Likewise, a formal eigensolution of $H$ at energy $E$ is a function $u\in \mathfrak D$ such that $\ell u = E u$ in the sense of equality of $L^1_\loc$ functions. Although $\sigma, \tau$ are prominent in these definitions, different choices of decomposition \eqref{HMdecomposition} lead to the same operator $H$ up to a change of the value of $\alpha$ \cite[Remark 2.2]{LSW}; the Dirichlet operator $\alpha=0$ is unchanged. As in the $L^1_\loc$ setting, these half-line operators have simple spectrum and a canonical spectral measure $\mu$. General criteria for spectral type were studied in \cite{LSW}.

A potential $V$ is said to be $H^{-1}$-decaying if 
\[
 \lVert V \phi(\cdot -n) \rVert_{H^{-1}} \to 0, \qquad n \to\infty.
\]
Due to \eqref{HMinequality}, we think of $\lVert \sigma \chi_{(x,x+1)} \rVert_2 +  \lVert \tau \chi_{(x,x+1)} \rVert_1$ as the local size of the potential, and for an $H^{-1}$-decaying potential, we assume that $\sigma, \tau$ are chosen so that
\[
\lVert \sigma \chi_{(x,x+1)} \rVert_2 +  \lVert \tau \chi_{(x,x+1)} \rVert_1  \to 0, \qquad x \to \infty.
\]
By a quadratic form argument \cite{LSW}, if $V$ is $H^{-1}$-decaying, $\sigma_\ess(H) = [0,\infty)$.  Finally, to describe rates of decay, we define spaces of half-line distributions
\[
\ell^p(H^{-1}) = \{ \sigma' + \tau \mid \sigma \in \ell^p(L^2), \tau \in \ell^p(L^1) \}.
\]

\begin{definition}
For an $H^{-1}$-decaying potential $V$, we say an eigensolution $u$ of $H_V$ at energy $E = k^2$ has WKB asymptotic behavior if 
\begin{equation}\label{WKB}
u(x) = e^{ik x - \frac{i}{2k} \int_0^x \tau(t)\,dt}  + o(1), \qquad x \to +\infty,
\end{equation}
\begin{equation}\label{WKBquasiderivative}
u^{[1]}(x) = ik e^{ik x - \frac{i}{2k} \int_0^x \tau(t)\,dt}  + o(1), \qquad x \to +\infty.
\end{equation}
\end{definition}

We explain in Lemma~\ref{lemmaWKBgaugechange} in what sense this is independent of decomposition.

This regime was not previously studied in the literature, so even the following short range result is new (although its spectral consequences were described in \cite{LSW}):

\begin{theorem}\label{thmWKBshortrange}
If $V \in \ell^1(H^{-1})$, then for every $E > 0$, there is an eigensolution with the WKB asymptotic behavior.
\end{theorem}

The main results of this paper are two theorems for potentials which decay at a slower rate; these are generalizations of results of Christ--Kiselev to the locally $H^{-1}$ norm. The first works with potentials in an $\ell^p(H^{-1})$ space:

\begin{theorem}\label{thmWKB}
If $V \in \ell^p(H^{-1})$ for some $p\in (1,2)$, then for Lebesgue-a.e.\ $E > 0$, there is an eigensolution with the WKB asymptotic behavior.  In particular, the absolutely continuous part of the Schr\"odinger operator $H$ is unitarily equivalent to the half-line Dirichlet Laplacian.
\end{theorem}

Combining this with some power law decay also bounds the Hausdorff dimension of the set of positive energies without WKB behavior: % singular part of the spectral measure:

\begin{theorem}\label{thmWKBhausdorff}
Let $p\in (1,2]$, $\gamma>0$ with $\gamma p'\lqs 1$, where $1/p + 1/p' = 1$. If $(1+x)^{\gamma}V(x) \in \ell^p(H^{-1})$, there exists a set $\Lambda$ of Hausdorff dimension $\dim_{\cH} \Lambda \lqs 1-\gamma p'$ such that for all $E\in (0,\infty)\setminus \Lambda$, there exists an eigensolution $Hu=Eu$ with the WKB asymptotic behavior. In particular, the singular part of the spectral measure of $H$ is supported on a set of Hausdorff dimension at most $1-\gamma p'$. 
\end{theorem}

We note that by H\"older's inequality, if $(1+x)^{\gamma}V(x) \in \ell^p(H^{-1})$, then $V \in \ell^r(H^{-1})$ for all $r > p / (1+p\gamma)$. In particular,  conclusions of Theorem~\ref{thmWKB} apply to the potentials of Theorem~\ref{thmWKBhausdorff}. Moreover, the case $\gamma p' > 1$ is already covered by Theorem~\ref{thmWKBshortrange}, since then $V \in \ell^1(H^{-1})$ by H\"older's inequality.

The first part of the analysis is pointwise in energy; it is a rewriting of the 2nd order ODE as a first-order vector ODE, with a change of variables accounting for the WKB asymptotics. This results in an initial value problem with an initial condition at infinity.  The resulting ODE has $L^1_\loc$ coefficients but a more complicated form than the classical case, with $\sigma, \tau$ appearing in different places and a nonlinearity in the form of a $\sigma^2$ term. An effective term replacing $V(x)$ in this initial value problem turns out to be the complex-valued, energy-dependent expression
\[
\tilde Q(x,E) = \tau(x) - \sigma(x)^2 + 2 i \sqrt E \sigma(x),
\]
which complicates further analysis.

The main part of the proof of Theorems~\ref{thmWKB}, \ref{thmWKBhausdorff} combines the original proof of Christ--Kiselev \cite{ChristKiselev98} with technical extensions introduced by Christ--Kiselev \cite{CAK2001CMP} in order to study linear combinations of terms with different decay properties; in our work, these extensions are used to handle energy-dependent linear combinations stemming from the effective potential $\tilde Q(x,E)$. Note that whereas \cite{CAK2001CMP} allows slowly decaying terms whose derivative is in an $L^p$ space, our work goes in the opposite direction and allows the potential to be a derivative. We also use some contributions of Liu~\cite{liu_multilinear}, who studied perturbations of periodic Schr\"odinger operators.

One motivation for Theorems~\ref{thmWKB}, \ref{thmWKBhausdorff} are potentials consisting of terms which are not locally integrable. 
For instance, the above theorems apply to combinations of $\delta$-functions
\[
V = \sum_{n=1}^\infty a_n \delta_n
\]
with a suitably decaying sequence of $a_n$. Another motivation is that fast oscillations make a potential appear smaller in $H^{-1}$ norm. For instance, a suitable potential of the form
\begin{equation}\label{eqnPotentialExample}
V(x) = g(x) \sin(x^b)
\end{equation}
where $g(x)$ behaves roughly as $x^a$, may appear to behave roughly as $x^{a+1-b}$ in local $H^{-1}$ norm, which is an improvement if $b >1$. We make this precise in the following example. Recall that a function $f:(0,\infty) \to (0,\infty)$ is said to be regularly varying (at $\infty$) of order $\rho$ if $f(\lambda x) / f(x) \to \lambda^\rho$ as $x \to\infty$ for every $\lambda > 0$.

\begin{example}\label{WKBoscillatoryexample}
Let $V$ be of the form \eqref{eqnPotentialExample}, where $g \in \AC_\loc((0,\infty))$ and  $g'$ is a regularly varying function of order $a-1$. Denote $c = b - a -1$. 
	\begin{itemize}
		\item [(a)] if $c>0$, then $V$ is $H^{-1}$-decaying, so $\sigma_{\ess}(H) = [0,\infty)$.
		\item [(b)] if $c>\frac{1}{2}$, then $V\in \ell^p(H^{-1})$ for $p\in (1/c,2)$, so by Theorem \ref{thmWKB}, $\sigma_{\ac}(H) = [0,\infty)$.
		\item [(c)] if $\frac{1}{2}<c\lqs 1$, then $(1+x)^{\gamma} V(x) \in \ell^2(H^{-1})$ for $\gamma \in (0,  c-1/2)$, so by Theorem \ref{thmWKBhausdorff},  $\dim_{H}(S) \le 2-2c$.
		\item [(d)] if $c>1$, then $V\in \ell^1(H^{-1})$, so $H$ has purely a.c. spectrum on $(0,\infty)$. 
	\end{itemize}
\end{example}

In the special case $g(x) = x^a$, more was already known, by an approach requiring $g$ to be infinitely differentiable with decay conditions on derivatives of all orders \cite{MR682723} (see also references therein).  Another example is the potential defined piecewise by
\begin{equation}\label{eqnPM1}
V(x) = (-1)^{2n\lfloor x - n\rfloor}, \qquad n-1 \le x < n, \quad n=1,2,3,\dots,
\end{equation}
 sometimes used as an example of a potential not decaying in a classical sense but having related properties \cite{EichingerLukic,LSW}. We obtain its spectral properties:

\begin{example}\label{xmplPM1}
The Schr\"odinger operator with potential given by \eqref{eqnPM1} %is of the form  $V(x) = \sigma'(x)$ with $\sigma(x) = O(1/x)$; by Theorems~\ref{thmWKB}, \ref{thmWKBhausdorff}, it
 has a.c.\ spectrum on $[0,\infty)$ and the singular part of its spectral measure is zero-dimensional.
\end{example}

%%%%%%%%%%%%%%%%%%%%%%%%%%%%%%%%%%%%%%%%
%%%%%%%%%%%%%%%%%%%%%%%%%%%%%%%%%%%%%%%%
\section{Observations about the decomposition of the potential}

A technicality of the $H^{-1}$ setup is that certain claims about the Schr\"odinger operators ostensibly depend on the choice of decomposition. We explain that WKB asymptotic behavior is only affected by a asymptotically constant phase shift, which can be factored out:

\begin{lemma}\label{lemmaWKBgaugechange}
If $V$ is $H^{-1}$-decaying and
\[
V = \sigma_k' + \tau_k, \qquad k = 1,2,
\]
are two distinct decompositions with the decay condition
\[
\int_j^{j+1} \lvert \sigma_k(x) \rvert^2 \,dx + \int_j^{j+1} \lvert \tau_k(x) \rvert \,dx  \to 0, \qquad j \to \infty,
\]
 then
\[
L = \lim_{x \to \infty} \int_0^x (\tau_1(t) - \tau_2(t))\,dt
\]
is convergent. In particular, if $u$ satisfies WKB asymptotic behavior with respect to $\sigma_1, \tau_1$, then $e^{iL/(2k)} u$ satisfies WKB asymptotic behavior with respect to $\sigma_2, \tau_2$.
\end{lemma}

\begin{proof}
Due to $(\sigma_2 - \sigma_1)' = \tau_1 - \tau_2$, the difference $\theta = \sigma_2 - \sigma_1$ is locally absolutely continuous. Since $\int_j^{j+1} \lvert \theta(x) \rvert^2\,dx \to 0$ and $\int_j^{j+1} \lvert \theta'(x) \rvert \,dx \to 0$, by a Sobolev inequality, $\theta(x) \to 0$ as $x \to\infty$. Convergence of the limit follows from $\int_0^x (\tau_1(t) - \tau_2(t))\,dt = \theta(x) - \theta(0)$.  Thus,
\[
e^{\frac{i}{2k} \int_0^x ( \tau_1(t) - \tau_2(t)) \,dt} = e^{\frac{iL}{2k}} + o(1), \qquad x \to\infty,
\]
and multiplying by WKB asymptotics for $u$ gives the final claim.
\end{proof}

\begin{lemma}\label{lpdomains}
	If $f\in \ell^p(L^2)$, then $f, f^2 \in \ell^p(L^1)$.
\end{lemma}

\begin{proof}
For any $j$, by the Cauchy--Schwarz inequality,
\[
\int_j^{j+1} |f(x)|\,\dd x \lqs \left(\int_j^{j+1}|f(x)|^2\,\dd x\right)^{1/2}\left(\int_j^{j+1}1\,\dd x\right)^{1/2} = \left(\int_j^{j+1}|f(x)|^2\,\dd x\right)^{1/2}.
\]
Taking $p$-th powers and summing in $j$ proves $f \in \ell^p(L^1)$.

By the well-known inclusion $\ell^p \subset \ell^q$ for $q > p$, $f \in \ell^p(L^2)$ implies $f \in \ell^{2p}(L^2)$. Note that $f \in \ell^{2p}(L^2)$  if and only if $f^2 \in \ell^p(L^1)$, since they both correspond to the convergence condition
\[
\sum_j \left( \int_j^{j+1} \lvert f(x) \rvert^2 \,dx \right)^p = \sum_j \left( \sqrt{ \int_j^{j+1} \lvert f(x) \rvert^2 \,dx } \right)^{2p} < \infty. \qedhere
\]
\end{proof}

\begin{lemma}\label{weightedEllpH-1space}
	Let $p \ge 1$, $\gamma \ge 0$, and $(1+x)^{\gamma}V\in \ell^p(H^{-1})$. Then $V$ has a decomposition $V=\sigma'+\tau$ such that
	\[
	(1+x)^{\gamma}\sigma\in \ell^p(L^2),\qquad (1+x)^{\gamma}\tau\in \ell^p(L^1).
	\]
	Moreover, for this decomposition, $(1+x)^\gamma (\tau - \sigma^2) \in \ell^p(L^1)$.
\end{lemma}

\begin{proof}
	By definition, there exist $a\in \ell^p(L^2)$ and $b\in \ell^p(L^1)$ such that
	\[
	(1+x)^{\gamma}V = a' + b\quad\implies\quad V = (1+x)^{-\gamma}\big(a' + b\big).
	\]
	Let $\sigma(x) = (1+x)^{-\gamma}a$, then $(1+x)^{\gamma}\sigma = a\in \ell^p(L^2)$; moreover,
	\[
	\sigma' = (1+x)^{-\gamma}a' - \gamma(1+x)^{-\gamma-1}a,
	\]
	and 
	\[
	V = \sigma' + \tau,\qquad \tau(x):= (1+x)^{-\gamma}\left(\frac{\gamma}{1+x}a(x) + b(x) \right).
	\]
	By Lemma~\ref{lpdomains},  $a \in \ell^p(L^2)$ implies $a \in \ell^p(L^1)$, and by the pointwise estimate
	\[
	\left\lvert \frac{a(x)}{1+x} \right\rvert \le \lvert a(x) \rvert,
	\]
this implies $(1+x)^{-1} a \in \ell^p(L^1)$. Moreover, $b\in \ell^p(L^1)$, so $(1+x)^\gamma \tau \in \ell^p(L^1)$.

	Applying Lemma \ref{lpdomains} to $(1+x)^{\gamma}\sigma$ implies 
	\[ (1+x)^{\gamma}\sigma\in \ell^p(L^1), \qquad (1+x)^{2\gamma}\sigma^2\in \ell^p(L^1).\] 
	Then a pointwise estimate $(1+x)^{\gamma}\sigma^2 \le (1+x)^{2\gamma}\sigma^2$ implies  $(1+x)^{\gamma}\sigma^2\in \ell^p(L^1)$.
\end{proof}

%%%%%%%%%%%%%%%%%%%%%%%%%%%%%%%%%%%%%%%%
%%%%%%%%%%%%%%%%%%%%%%%%%%%%%%%%%%%%%%%%
\section{A pointwise condition for WKB asymptotic behavior}

We provide a  condition for the existence of a solution with WKB asymptotic behavior at a fixed energy $E$. The eigensolution equation $\ell u = Eu$ can be written as a first-order matrix ODE with $L^1_\loc$ coefficients,
\begin{equation}\label{eigenfunctionMatrix}
\begin{pmatrix}
u^{[1]} \\
u
\end{pmatrix}'
 =  \begin{pmatrix}
-  \sigma & \tau - \sigma^2 - E  \\
1 &  \sigma
\end{pmatrix} 
\begin{pmatrix}
u^{[1]} \\
u
\end{pmatrix},
\end{equation}
and the proof consists of transforming this ODE into another one.

\begin{theorem}\label{equivSolutions}
Fix $\sigma, \tau$ and fix $E > 0$. Denote $k= \sqrt E$ and
\begin{equation}\label{funcDefs}
	\begin{split}
		Q(x) & = \tau(x) - \sigma(x)^2,\qquad \tilde Q(x,E) = Q(x) + 2 i k \sigma(x)  \\
		h(x,E) & = 2 kx - \int_0^x \frac{Q(t)}{k} \,dt,\qquad w(x,E) = - \frac i{2k} \\
		&\quad \cF(x,E)= w(x,E) e^{-ih(x,E)} \tilde Q(x,E).
	\end{split}
\end{equation}
If the system
\begin{equation}\label{systemInY}
Y'(x) = D(x,E)
Y(x), \qquad D(x,E)=\begin{pmatrix}
0 & \cF(x,E) \\
\ol{\cF(x,E)} & 0
\end{pmatrix}
\end{equation}
has a solution obeying
\begin{equation}\label{WKBrestated2}
Y(x) = \begin{pmatrix}
1 \\
0
\end{pmatrix} + o(1), \qquad x \to \infty
\end{equation}
then there is an eigensolution $u$ obeying the asymptotic behavior
\begin{equation}\label{WKB2}
u(x) = e^{ik x - \frac{i}{2k} \int_0^x ( \tau - \sigma^2)(t)\,dt}  + o(1), \qquad x \to +\infty,
\end{equation}
\begin{equation}\label{WKBquasiderivative2}
u^{[1]}(x) = ik e^{ik x - \frac{i}{2k} \int_0^x ( \tau - \sigma^2)(t)\,dt}  + o(1), \qquad x \to +\infty.
\end{equation}
In particular, if $\sigma \in \ell^2(L^2) = L^2((0,\infty))$, then there is an eigensolution obeying the WKB asymptotic behavior  \eqref{WKB}, \eqref{WKBquasiderivative}.
\end{theorem}

\begin{proof}
With the substitution
\[
u_2 = \begin{pmatrix}
e^{ih/2} & 0 \\
0 & e^{-ih/2}
\end{pmatrix} Y,
\]
we obtain $u_2$ which obeys the ODE
\begin{align*}
u_2' & = \begin{pmatrix}
ih'/2 & 0 \\
0 & -ih'/2
\end{pmatrix}
\begin{pmatrix}
e^{ih/2} & 0 \\
0 & e^{-ih/2}
\end{pmatrix} Y +  \begin{pmatrix}
e^{ih/2} & 0 \\
0 & e^{-ih/2}
\end{pmatrix} \begin{pmatrix}
0 & \cF \\
\ol{\cF} & 0
\end{pmatrix}Y \\
& =
\begin{pmatrix}
ih'/2 & e^{ih} \cF \\
e^{-ih} \ol{\cF} & -ih'/2
\end{pmatrix} u_2
\end{align*}
and with the further substitution
\[
u_1 = \begin{pmatrix}
i k & -ik \\
1 & 1
\end{pmatrix} u_2
\]
where $k=\sqrt{E}$, this gives $u_1$ which obeys the ODE
\[
u_1' =
 \begin{pmatrix}
i k & -ik \\
1 & 1
\end{pmatrix} 
\begin{pmatrix}
ih'/2 & e^{ih} \cF \\
e^{-ih} \ol{\cF} & -ih'/2
\end{pmatrix}
 \begin{pmatrix}
i k & -ik \\
1 & 1
\end{pmatrix}^{-1} u_1
\]
By direct calculations, this gives
\begin{align*}
u_1'  & =  \begin{pmatrix}
- \frac 1{2k} \Im \tilde Q & - \frac{kh'}2 + \frac 12 \Re \tilde Q \\
\frac{h'}{2k} +  \frac 1{2k^2} \Re \tilde Q &  \frac 1{2k} \Im \tilde Q 
\end{pmatrix} 
u_1 \\
 & =  \begin{pmatrix}
-  \sigma & Q - k^2  \\
1 &  \sigma
\end{pmatrix} 
u_1 
\end{align*}
and we recognize this as the matrix ODE for eigenfunctions \eqref{eigenfunctionMatrix}.

Moreover, from the asymptotic behavior $Y = \binom 10 + o(1)$, since $\lvert e^{ih/2} \rvert = 1$ we obtain
\[
\left\lVert u_2 - \begin{pmatrix}
e^{ih/2} \\
0
\end{pmatrix}
\right\rVert
= \left\lVert Y - \begin{pmatrix}
1 \\
0
\end{pmatrix}
\right\rVert
\to 0, \qquad x \to \infty
\]
and then, since $(\begin{smallmatrix} ik & -ik \\ 1 & 1 \end{smallmatrix})$ is a fixed invertible matrix, we obtain
\[
\left\lVert
u_1 - \begin{pmatrix} 
ik e^{ih/2} \\
e^{ih/2}
\end{pmatrix}
\right\rVert 
\le 
\left\lVert  \begin{pmatrix}
i k & -ik \\
1 & 1
\end{pmatrix} 
\right\rVert
\left\lVert
u_2 - \begin{pmatrix} 
e^{ih/2} \\
e^{ih/2}
\end{pmatrix}
\right\rVert 
\to 0, \qquad x\to \infty
\]
and therefore
\[
u_1 = \begin{pmatrix} 
ik e^{ih/2} + o(1) \\
e^{ih/2} + o(1)
\end{pmatrix}, \qquad x \to\infty.
\]
This is equivalent to \eqref{WKB2}, \eqref{WKBquasiderivative2}. If $\sigma \in \ell^2(L^2)$, this can be transformed to  \eqref{WKB}, \eqref{WKBquasiderivative} as in the proof of Lemma~\ref{lemmaWKBgaugechange}.
\end{proof}

If $\sigma \notin \ell^2(H^{-1})$, the appearance of $\sigma^2$ in the exponents in \eqref{WKB2}, \eqref{WKBquasiderivative2} is of interest. Potentials decaying slower than $L^2$ may have empty a.c.\ spectrum \cite{KotaniUshiroya,MR1628290}, so no WKB-type behavior can be expected in general, but in some slowly decaying settings, such as Wigner--von Neumann type potentials with decay slower than $L^2$ \cite{JanasSimonov10,Lukic11,Lukic13,Lukic14,Simonov16,Gwaltney24}, precise asymptotics with additional correction terms are obtained, and those additional terms are of quadratic and higher orders in the potential.

At this point, the "short-range" case $V \in \ell^1(H^{-1})$ follows immediately (see also \cite{LSW}):

\begin{corollary} \label{corShortRangeCase}
If $V \in \ell^1(H^{-1})$, then for every $E > 0$, there exists an eigensolution with WKB asymptotic behavior.
\end{corollary}

\begin{proof}
By Lemma~\ref{weightedEllpH-1space}, there exists a decomposition such that $\sigma, \tau - \sigma^2 \in \ell^1(L^1) = L^1((0,\infty))$ so $\cF(x,E) \in L^1((0,\infty))$. Thus, for any $E> 0$, there is a solution $Y$ of \eqref{systemInY}, \eqref{WKBrestated2} given by a classical Volterra series.
\end{proof}

WKB asymptotic behavior provides further information about the eigensolutions: 

\begin{lemma}\label{oneWKBimpliesALL}
If at some $E>0$ there is an eigensolution $u$ which obeys \eqref{WKB}, \eqref{WKBquasiderivative}, then $u$ and $\ol u$ are linearly independent, and all eigensolutions at $E$ are bounded.
\end{lemma}

\begin{proof}
The Wronskian of two eigensolutions $u,v$ is
\[
W(u,v) (x) = u(x) v^{[1]}(x) - u^{[1]}(x) v(x).
\]
The Wronskian of $u, \ol u$ is independent of $x$, and from \eqref{WKB}, \eqref{WKBquasiderivative}, it follows that
\[
W(\ol u, u)(x) = 2ik  +o (1), \qquad x \to \infty
\] 
so $W(\ol u, u) = 2ik$. In particular, $\ol u, u$ are linearly independent.

Any eigensolution at energy $E$ is a linear combination of $\ol u, u$, so it is bounded.
\end{proof}

Transfer matrices $T(x,z)$ are obtained as the matrix solution of the initial value problem
\[
\partial_x T(x,z) = \begin{pmatrix}
-\sigma(x) & \tau(x) - \sigma(x)^2 - z \\
1 & \sigma(x)
\end{pmatrix}
T(x,z), \qquad T(0,z) = I = \begin{pmatrix} 1 & 0 \\ 0 & 1 \end{pmatrix}
\]
derived from \eqref{eigenfunctionMatrix}. This is intended as the unique solution which is locally absolutely continuous in $x$ for every $z$.

A nontrivial eigensolution $u$ is called subordinate if for every eigensolution $v$ linearly independent to $u$,
\[
\lim_{x\to\infty} \frac{ \int_0^x \lvert u(t) \rvert^2 \,dt } { \int_0^x \lvert v(t) \rvert^2 \,dt }  = 0.
\]
Next, we note the very general statement that boundedness of solutions implies absence of subordinate solutions and absolute continuity of the spectral measure. In the classical setting, this is a combination of results of Stolz \cite{Stolz92} with the subordinacy theory of Gilbert--Pearson \cite{MR915965}; in the $H^{-1}_\loc$ setting, this follows from the arguments of \cite{LSW}, and we provide the steps of the proof not explicitly stated there:

\begin{lemma}\label{lemmaSubordinate}
Assume $\sigma \in \ell^\infty(L^2)$, $\tau \in \ell^\infty(L^1)$. If all eigensolutions are bounded at some energy $E\in \bbR$, then there are no subordinate solutions at energy $E$.

In particular, if eigensolutions are bounded for all $E$ in a Borel set $S$ and $\mu$ denotes the canonical spectral measure of $H$, then $\chi_S \,d\mu$ is mutually absolutely continuous with $\chi_S \,dm$, where $m$ is Lebesgue measure.
\end{lemma}

\begin{proof}
In \cite[Proof of Theorem 1.3]{LSW}, it was proved that existence of a subordinate solution implies
\[
\lim_{l \to \infty} \frac 1l \int_0^l \lVert T(x,E)\rVert^2\,dx = \infty.
\]
However, boundedness of eigensolutions implies boundedness of their quasiderivatives by the eigensolution estimates \cite[Lemma 2.7]{LSW}, so it implies $\sup_x \lVert T(x,E) \rVert < \infty$. Combining the two, we see that boundedness of eigensolutions implies that there is no subordinate solution.

By subordinacy theory (\cite{MR1761504,EichingerLukicSimanek} in this generality), the set $N \subset \bbR$ of energies at which there is no subordinate solution is, up to a set of measure zero, equal to the set of energies $E \in \bbR$ at which
\[
\lim_{\epsilon \downarrow 0} m(E+i\epsilon) \in \bbC_+.
\]
By the general properties of the Herglotz representation, it follows that $\chi_N \, d\mu$ is mutually absolutely continuous with $\chi_N \,dm$, with $m$ the Lebesgue measure.
\end{proof}

%%%%%%%%%%%%%%%%%%%%%%%%%%%%%%%%%%%%%%%%
%%%%%%%%%%%%%%%%%%%%%%%%%%%%%%%%%%%%%%%%
\section{Martingale structures and operator estimates}

Before we proceed to the proofs of Theorem \ref{thmWKB} and Theorem \ref{thmWKBhausdorff}, we need some preliminary notions and results. Let us start by introducing the martingale structure :

\begin{definition}
A collection of subintervals $\{E^m_j: m\in\bbZ_+, 1\lqs j \lqs 2^m\}$ is called a martingale structure on $\bbR_+$ if the following is true \cite{ChristKiselev, liu_multilinear}:
	\begin{itemize}
		\item $\forall m$, $\bbR_+ = \cup_j E^m_j$;
		\item $\forall i \neq j$, $E^m_i\cap E^m_j =\emptyset$;
		\item If $i<j$, $x\in E^m_i$ and $x'\in E^m_j$, then $x<x'$;
		\item $\forall m$, $E^m_j = E^{m+1}_{2j-1}\cup E^{m+1}_{2j}$.
	\end{itemize}
	Given a martingale structure $\{E^m_j\}$, let $\chi^m_j:= \chi_{E^m_j}$; the martingale structure is said to be adapted (in $\ell^p(L^1)$) to $f$ if for all $m,j$:
	\[
	\|f\chi^m_j\|^p_{\ell^p(L^1)} \lqs 2^{-m}\|f\|_{\ell^p(L^1)}^p.
	\]
\end{definition}

\begin{lemma}[{\cite[p.433]{ChristKiselev}}]
	For any function $f\in \ell^p(L^1)$, there exists a martingale structure $\{E^m_j\}$ adapted to $f$.
\end{lemma}

Next, we introduce the $\mathcal B_s$ semi-norm which will be an important object throughout the section.
\begin{definition}
	For $s>0$, let $\cB_s$ be the Banach space consisting of all $\bbC$-valued sequences $a=a(m,j)$, $m\in\bbZ_+$ and $1\lqs j\lqs 2^m$, for which
	\begin{equation}
		\|a\|_{\mathcal{B}_s} = \sum_{m\in\bbZ_+}m^s\left( \sum_{j=1}^{2^m} |a(m,j)|^2\right)^{1/2} <\infty.
	\end{equation}
\end{definition}

\begin{definition}
For $f\in L^1_{\loc}(\bbR_+)$ and a martingale structure $\{E^m_j\}$, define a sequence
\begin{equation}
	\left\{ \int_{E^m_j} f(x)\,\dd x \right\} = \left\{ \int_{\bbR_+} f(x)\chi^m_j(x)\,\dd x \right\}.
\end{equation}
By abusing the notation, we denote
\begin{equation}
	\|f\|_{\mathcal{B}_s} = \left\| \left\{ \int_{E^m_j} f(x)\,\dd x \right\} \right\|_{\mathcal{B}_s} = \sum_{m=1}^{\infty} m^s\left( \sum_{j=1}^{2^m}\left| \int_{E^m_j}f(x)\,\dd x \right|^2 \right)^{1/2}.
\end{equation}
\end{definition}

\begin{lemma}\label{B-seminorm}
	$\lVert \cdot \rVert_{\cB_s}$ is a semi-norm on the set
\[
\cB_s = \{ f \in L^1_{\loc}(\bbR_+) \mid \lVert f \rVert_{\cB_s} < \infty\}.
\]
\end{lemma}
\begin{proof}
	This follows from the Minkowski's inequality.
\end{proof}

Note that $\|f\|_{\cB_s}$ always assumes some underlying martingale structure $\{E^m_j\}$, though it is not necessarily adapted to $f$.

\begin{definition}
	Let $P_i:\ell^p(L^1)\to L^q(J,\dd E)$, $i=1,2$ be linear or sub-linear bounded operators, where $J\subset \bbR$ is a closed interval. For $s>0$ and a martingale structure $\{E^m_j\}$, define
\[
G^{(s)}_{P_1(f)}(E) = \|\{P_1(f\chi^m_j)(E)\}\}\|_{\mathcal{B}_s} = \sum_{m=1}^{\infty} m^s\left( \sum_{j=1}^{2^m}\left| P_1(f\chi^m_j)(E) \right|^2 \right)^{1/2},
\]
and
\begin{equation}
	\begin{split}
		G_{P_1(f),P_2(f)}^{(s)}(E) & = \left\| \left\{ P_1(f_1\chi^m_j)(E) + P_2(f_2\chi^m_j)(E)|^2 \right\} \right\|_{\mathcal{B}_s} \\
	& = \sum_{m=1}^{\infty} m^s\left( \sum_{j=1}^{2^m}|P_1(f_1\chi^m_j)(E) + P_2(f_2\chi^m_j)(E)|^2 \right)^{1/2}.
	\end{split}
\end{equation}
\end{definition}

Upon assuming boundedness of $P_i$, $i=1,2$, it can be shown that $G^{(s)}_{P_1(f)}$ and $G^{(s)}_{P_1(f),P_2(f)}$ are all in $L^q(J,\dd E)$:

\begin{lemma}[{\cite[Proposition 3.3]{ChristKiselev}}]\label{ck-JFA-proposition3.3}
	Assume that $P_i$, $i=1,2$ are bounded linear or sublinear operators from $\ell^p(L^1)$ to $L^q(J,\dd E)$, where $p<2<q$ and $J\subset \bbR$ is a closed interval. Then, 
	\begin{itemize}
		\item[(a)] for any $f\in \ell^p(L^1)$ and any martingale structure $\{E^m_j\}$ adapted to $f$,
			\[
			\|G^{(s)}_{P(f)}(E)\|_{L^q(J)}\lqs C(p,q,s,\|P\|)\cdot\|f\|_{\ell^p(L^1)};
			\]
		\item[(b)] for any $f_i\in \ell^p(L^1)$, $i=1,2$, and any martingale structure $\{E^m_j\}$ adapted to $|f_1|+|f_2|$, 
			\[ 
			\|G^{(s)}_{P_1(f_1), P_2(f_2)}(E)\|_{L^q(J)}\lqs C(p,q,s,\|P_1\|,\|P_2\|)\cdot(\|f_1\|_{\ell^p(L^1)} + \|f_2\|_{\ell^p(L^1)}),
			\]
	\end{itemize}
	where $C<\infty$ depends only on $p$, $q$ and the operator norm of $P$, or of $P_1$, $P_2$ respectively.
\end{lemma}

For an integral operator $P$ given by
\[ P(f)(E) = \int_{\bbR_+}p(x,E)f(x)\,\dd x \]
for some measurable function $p(x,E)$ on $\bbR_+\times J$, we define the maximal operator $P^*$ as
\[ P^*(f)(E)=\sup_{y\in\bbR_+}\left| \int_y^{\infty}p(x,E)f(x)\,\dd x \right|. \]

\begin{lemma}[{\cite[Lemma 3.4]{ChristKiselev}}]\label{ck-JFA-lemma3.4-maximalOp}
	The mapping $(f_1,f_2)\mapsto G_{P_1^*(f_1),P_2^*(f_2)}$ also satisfies the conclusion of Lemma \ref{ck-JFA-proposition3.3}.
\end{lemma}

We note that the proof to Lemma \ref{ck-JFA-lemma3.4-maximalOp} is identical to that of Lemma \ref{ck-JFA-proposition3.3} upon showing that the maximal operator $P^*$ of a bounded linear or sublinear operator $P$ is also bounded. 

For suitable functions $\zeta:\bbR_+ \times \bbR_+ \to \bbC$, we will consider the operator $S_{\zeta}$ and its maximal operator $S_\zeta^*$ on $\ell^p(L^1)$, given by
\begin{equation}\label{def-SandS*}
\begin{split}
	S_{\zeta}(f)(E) & = \int_{\bbR_+}\zeta(x,E)e^{-ih(x,E)}f(x)\,\dd x, \\
	(S_{\zeta})^*(f)(E) & = \sup_{y\in\bbR_+}\left| \int_y^{\infty}\zeta(x,E)e^{-ih(x,E)}f(x)\,\dd x \right|,
\end{split} 
\end{equation}

\begin{lemma}\label{liu-lemma2.8}
Let $p\in [1,2]$ and $p'=\frac{p}{p-1}$ be the conjugate of $p$ ($p'=\infty$ when $p=1$). If $\tilde J \subset \bbR_+$ is a compact interval and $J \subset \intt \tilde J$ a compact interval, and $\zeta$ obeys
\begin{equation}\label{cond-Zeta}
	\sup_{x\in\mathbb R_+,E\in \tilde{J}}\sum_{i=1}^2|\partial^i_E\zeta(x,E)|\lqs C_1,
\end{equation}
for some constant $C<\infty$, then $S_\zeta, S_\zeta^*: \ell^p(L^1) \to L^{p'}(J,\dd E)$ given by \eqref{def-SandS*} are well-defined bounded operators;  
for any $f\in \ell^p(L^1)$,
	\begin{equation}\label{opS-bound}
		\|S_{\zeta}(f)\|_{L^{p'}(J,\dd E)} \lqs  C \|f\|_{\ell^p(L^1)},\quad \| S_{\zeta}^*(f)\|_{L^{p'}(J,\dd E)} \lqs C \|f\|_{\ell^p(L^1)},
	\end{equation} 
	where $C$ depends on $J, \tilde{J},p$, and $C_1$ is as in \eqref{cond-Zeta}.
\end{lemma}

\begin{proof}
	The definition of $h(x,E)$ in \eqref{funcDefs} is exactly as that in \cite{liu_multilinear}, i.e., $Q\in \ell^p(L^1)$, and so the result of \cite[Theorem 4.1]{liu_multilinear} holds. Then, the claim follows \cite[Theorem 4.1]{liu_multilinear} with $h=h$ and $w=\zeta$.
\end{proof}

In our application in Section 6--7, $\zeta(E) = \zeta(x,E)$ will be of the form $ak^b$ for some constant $a\in \mathbb C$ and $b\in \mathbb R$; thus, the condition \eqref{cond-Zeta} will always hold for a fixed compact interval $\tilde{J}\subset\mathbb R_+$.

Finally, we introduce the multilinear operator technique considered in \cite{ChristKiselev, CAK2001CMP}. The multilinear operator will be used to define the series solution to the system \eqref{systemInY}, so the following results are essential in later proofs of the WKB asymptotics of eigensolutions.

\begin{definition}\label{multilinear-def}
	A multilinear operator $M_n$ acts on $n$ functions $f_j(x,E)$ by
	\[
	M_n(f_1,\dots,f_n)(x,x',E) = \int_{x \le t_1 \le \dots \le t_n \le x'} \prod_{j=1}^n f_j(t_j, E) \,dt_j.
	\]
	If there exists $f$ such that $f_j\in \{f, \overline{f}\}$ for all $j$, then by abusing the notation, we write 
\[
M_n(f)(x,x', E) = M_n(f_1,\ldots, f_n)(x,x', E).
\] 
\end{definition}

Note that even though the single notation $M_n(f)$ is used for all possible $\{f_1,\ldots, f_n\}\in \{f, \overline{f}\}^n$, there should be no ambiguity since $(f_1,\ldots, f_n)$ will be fixed in the proofs.

\begin{definition}
	The corresponding maximal operators are defined as
	\[
	M_n^*(f_1,\ldots,f_n)(E) = \sup_{x\lqs x'\in\mathbb R_+}\left| M_n(f_1,\ldots, f_n)(x,x',E) \right|.
	\]
	In the case there exists $f$ such that $f_j\in\{f,\overline{f}\}$ for all $j$,
	\[
	M_n^*(f)(E) = \sup_{0<x\lqs x'<\infty}|M_n(f)(x,x', E)|.
	\]
\end{definition}

\begin{theorem}[Christ--Kiselev \cite{Christ2001MaximalFA}]\label{multilinear-christ-kiselev-lemma}
	There exists a universal constant $C_0<\infty$ such that for every martingale structure $\{E^m_j\}$ on $\mathbb R_+$, 
	\[ M_n^*(f_1,\ldots,f_n)(E)\lqs C_0^n\prod_{j=1}^n\|f_j(\cdot, E)\|_{\mathcal B_1}, \]
	and
	\begin{equation}
		M_n^*(f)(E)\lqs C_0^n\frac{\|f(\cdot, E)\|^n_{\mathcal B_1}}{\sqrt{n!}}.
	\end{equation}
\end{theorem} 

\begin{remark}
	The notations $\|f\|_{\mathcal{B}_s}$ and $G^{(s)}_{P_1(f_1), P_2(f_2)}$ will always assume a certain underlying martingale structure $\{E^m_j\}$.
\end{remark}

%%%%%%%%%%%%%%%%%%%%%%%%%%%%%%%%%%%%%%%%
%%%%%%%%%%%%%%%%%%%%%%%%%%%%%%%%%%%%%%%%
\section{A series solution and summary of proof}
By Theorem \ref{equivSolutions}, in order to show Theorem \ref{thmWKB} and Theorem \ref{thmWKBhausdorff}, it suffices to produce a solution $Y(x)$ of the system in \eqref{systemInY} satisfying \eqref{WKBrestated2} on some appropriate subset $\mathcal O\subseteq\mathbb R_+$.

We solve the system in \eqref{systemInY} by solving the corresponding integral system:
	\[ Y(x) = \begin{pmatrix}
		1 \\ 0
	\end{pmatrix} -\int_x^{\infty} D(y)Y(y)\,\dd y. \]

By iterations, we obtain a series solution:
\begin{align*}
	Y(x) & = \begin{pmatrix}
		1 \\ 0
	\end{pmatrix} - \int_x^{\infty}D(y)\begin{pmatrix}
		1 \\ 0
	\end{pmatrix}\,\dd y  + \int_x^{\infty}\int_y^{\infty}D(y)D(t)Y(t)\,\dd t\,\dd y \\
	& = \begin{pmatrix}
		1 \\ 0
	\end{pmatrix} - \int_x^{\infty}D(y)\begin{pmatrix}
		1 \\ 0
	\end{pmatrix}\,\dd y  + \int_x^{\infty}\int_y^{\infty}D(y)D(t)\begin{pmatrix}
		1 \\ 0
	\end{pmatrix}\,\dd t\,\dd y \\ 
	& \qquad\qquad \qquad - \int_x^{\infty}\int_y^{\infty}\int_t^{\infty}D(y)D(t)D(s)Y(s)\,\dd s\,\dd t\,\dd y\\
	& \,\,\,\, \vdots \qquad\qquad\qquad\qquad\qquad\qquad \vdots
\end{align*}
\begin{equation}\label{seriesSolution1}
	\qquad\qquad\,\,\,\, Y(x) = \begin{pmatrix}
		1 \\ 0
	\end{pmatrix} + \sum_{k=1}^{\infty}(-1)^k\int\cdots\int_{x\lqs t_1\lqs \cdots\lqs t_k<\infty} D(t_1)\cdots D(t_k)\begin{pmatrix}
		1\\0
	\end{pmatrix}\dd t_k\cdots\dd t_1.
\end{equation}
Note that for even number of multiplications,
\[ D(t_1)\cdots D(t_{2k}) = \begin{pmatrix}
	\mathcal{F}(t_1)\overline{\mathcal{F}}(t_2)\mathcal{F}(t_3)\cdots \overline{\mathcal{F}}(t_{2k}) & 0 \\
	0 & \overline{\mathcal{F}}(t_1)\mathcal{F}(t_2)\overline{\mathcal{F}}(t_3)\cdots \mathcal{F}(t_{2k})
\end{pmatrix}, \]
and for odd number of multiplications,
\[ D(t_1)\cdots D(t_{2k+1}) = \begin{pmatrix}
	0 & \mathcal{F}(t_1)\overline{\mathcal{F}}(t_2)\mathcal{F}(t_3)\cdots \overline{\mathcal{F}}(t_{2k})\mathcal{F}(t_{2k+1}) \\
	\overline{\mathcal{F}}(t_1)\mathcal{F}(t_2)\overline{\mathcal{F}}(t_3)\cdots \mathcal{F}(t_{2k})\overline{\mathcal{F}}(t_{2k+1}) & 0
\end{pmatrix}; \]
thus, using the multilinear operator notation from Definition \ref{multilinear-def},
\[\int\cdots\int_{x\lqs t_1\lqs \cdots\lqs t_{2k}\lqs x'} D(t_1)\cdots D(t_{2k})\begin{pmatrix}
		1\\0
	\end{pmatrix}\dd t_{2k}\cdots\dd t_1 = \begin{pmatrix}
		M_{2k}(\mathcal{F})(x,x', E)\\ 0
	\end{pmatrix},  \]
\[\int\cdots\int_{x\lqs t_1\lqs \cdots\lqs t_{2k+1}\lqs x'} D(t_1)\cdots D(t_{2k+1})\begin{pmatrix}
		1\\0
	\end{pmatrix}\dd t_{2k+1}\cdots\dd t_1 = \begin{pmatrix}
		0 \\ M_{2k+1}(\mathcal{F})(x,x',E)
	\end{pmatrix},  \]
So, the series solution \eqref{seriesSolution1} becomes
\begin{equation}\label{seriesSolution2}
	Y(x) = \begin{pmatrix}
		1 \\ 0
	\end{pmatrix} + \begin{pmatrix}
		\sum_{m=1}^{\infty}M_{2m}(\mathcal{F})(x,\infty,E) \\
		-\sum_{m=0}^{\infty}M_{2m+1}(\mathcal{F})(x,\infty,E)
	\end{pmatrix}.
\end{equation}

Now, it suffices to show that, on some appropriate subset $\mathcal O$ of $\mathbb R_+$, the series in \eqref{seriesSolution2} is well-defined and gives an actual solution to \eqref{systemInY} which satisfies \eqref{WKBrestated2}. 

For Theorem \ref{thmWKB}, the goal is to show that $\mathcal O\subseteq\mathbb R_+$ is a full Lebesgue measure set; as for Theorem \ref{thmWKBhausdorff}, we will show that $\mathcal O=\mathbb R_+\setminus \Lambda$ for some $\Lambda$ with Hausdorff dimension less or equal to $1-\gamma p'$. The proofs will rely on the following general theorem:

\begin{theorem}\label{thm2.5Liu}
	Assume that for $j=1,\ldots, n$, $f_j\in L^1_{\loc}(\mathbb R_+)$. Suppose that there exists a constant $C$ (does not depend on $I$) such that for any closed interval $I\subset\mathbb R_+$,
	\begin{equation}\label{cond1-general}
		\|f_j\chi_I\|_{\mathcal B_1}\lqs C,
	\end{equation}
	and suppose 
	\begin{equation}\label{cond2-general}
		\limsup_{M\to\infty}\|f_j\chi_{[M,\infty)}\|_{\mathcal B_1} = 0.
	\end{equation}
	Then, the limit 
	\begin{equation}\label{multilinearLimit}
		B_n(f_1,\ldots,f_n)(x):=\lim_{y_1,\ldots,y_n\to\infty}\int_x^{y_1}\int_{t_1}^{y_2}\cdots\int_{t_{n-1}}^{y_n}\prod_{j=1}^nf_j(t_j)\,\dd t_1\dd t_2\cdots\dd t_n
	\end{equation}
	is well-defined, and
	\begin{equation}\label{forVerifyingWKB}
		\lim_{x\to\infty} B_n(f_1,\ldots,f_n)(x) = 0.
	\end{equation} 
	Moreover, $B_n(f_1,\ldots,f_n)\in \AC_{\loc}(\mathbb R_+)$ and for almost every $x$,
	\begin{equation}\label{forCheckingDEsystem}
		\frac{\dd B_n(f_1,\ldots,f_n)(x)}{\dd x} = -f_1B_{n-1}(f_2,\ldots,f_n)(x).
	\end{equation}
\end{theorem}

\begin{proof}
	All of the above claims, except for the claim $B_n(f_1,\ldots,f_n)\in \AC_{\loc}(\mathbb R_+)$, are explicitly stated in \cite[Theorem 2.5]{liu_multilinear}. The claim that $B_n(f_1,\ldots,f_n)\in \AC_{\loc}(\mathbb R_+)$ can be proved by induction in $n$. Follow along the lines of the proof of \cite[Theorem 2.5]{liu_multilinear}, for any $0<x<y$,
	\[ B_1(f_1)(y) - B_1(f_1)(x) = \int_y^{\infty}f_1(t_1)\,\dd t_1 - \int_x^{\infty}f_1(t_1)\,\dd t_1 = -\int_x^yf_1(t_1)\,\dd t_1 \]
	where $f_1\in L^1_{\loc}(\mathbb R_+)$ and thus $B_1(f_1)\in \AC_{\loc}(\mathbb R_+)$. Next, consider
	\begin{align*}
		B_n(f_1,\ldots,f_n)(y) - B_n(f_1,\ldots,f_n)(x) & = \int_y^xf_1(t_1)\,\dd t_1\int_{t_1}^{\infty}\cdots\int_{t_{n-1}}^{\infty}\prod_{j=2}^nf_j(t_j)\,\dd t_2\cdots\dd t_n \\
		& = -\int_x^yf_1(t_1)B_{n-1}(f_2,\ldots,f_n)(t_1)\,\dd t_1,
	\end{align*}
	where $f_1\in L^1_{\loc}(\mathbb R_+)$ and $B_{n-1}(f_2,\ldots,f_n)\in\AC_{\loc}(\mathbb R_+)$ by induction hypothesis; in particular, $B_{n-1}(f_2,\ldots,f_n)$ is locally bounded, so
	\[ -f_1B_{n-1}(f_2,\ldots,f_n)\in L^1_{\loc}(\mathbb R_+) \] 
	and therefore $B_n(f_1,\ldots,f_n)\in \AC_{\loc}(\mathbb R_+)$.
\end{proof}

We followed the approach in \cite{liu_multilinear} to define the multilinear operators as iterations as opposed to the Christ--Kiselev's approach in \cite[Proposition 4.2]{ChristKiselev}, which relies only on the existence of a weaker limit
\[
\lim_{y\to\infty}\int_x^y\int_{t_1}^y\cdots\int_{t_{n-1}}^y\prod_{j=1}^nf_j(t_j)\,\dd t_1\dd t_2\cdots\dd t_n
\]
in comparison to the limit in \eqref{multilinearLimit}.

To illustrate how to use Theorem \ref{thm2.5Liu} to prove Theorem \ref{thmWKB} and Theorem \ref{thmWKBhausdorff}, note that these follow by applying the following criterion for a large enough set of positive energies $E$:

\begin{lemma}\label{lemmasummaryOfProof}
If for some $E > 0$,
\begin{equation}\label{cond2}
\limsup_{M\to\infty} \|\mathcal F(\cdot,E)\chi_{[M,\infty)}\|_{\mathcal B_1} = 0,
\end{equation}
then there exists an eigensolution at energy $E$ with WKB asymptotic behavior.
\end{lemma}

\begin{proof}
Note that \cite{liu_multilinear} assumes a second condition that
\begin{equation}\label{cond1}
\|\mathcal F(\cdot, E)\chi_I\|_{\mathcal B_1}\lqs C(E)
\end{equation}
for every interval $I$, with a constant $C(E)$ independent of $I$. However, \eqref{cond2} implies existence of $C_1$ such that for every $x \ge C_1$,
\[
\|\mathcal F(\cdot,E)\chi_{[x,\infty)}\|_{\mathcal B_1} \le 1.
\]
Since $\lVert \cdot \rVert_{\cB_1}$ is a seminorm, this implies for every interval $I = [x,y] \subset [C_1, \infty)$ that \eqref{cond1} holds, 
with an explicit constant $C(E) = 2$. The rest of this proof can be done on the interval $[C_1, \infty)$; the eigensolution then extends to $[0,\infty)$, and if $u$ obeys
\[
u(x) = e^{ik(x-C_1) - \frac{i}{2k} \int_{C_1}^x \tau(t)\,dt} + o(1), \qquad x \to \infty
\]
then the eigensolution  $e^{i\phi} u$, $\phi = {kC_1 - \frac 1{2k} \int_0^{C_1} \tau(t) \,dt}$, obeys \eqref{WKB}, and similarly \eqref{WKBquasiderivative}.

Note that the conditions \eqref{cond1}, \eqref{cond2} correspond respectively to the assumptions \eqref{cond1-general}, \eqref{cond2-general} in Theorem \ref{thm2.5Liu}. Firstly, \eqref{multilinearLimit} applied to the current scenario implies that the limit
	\[
	M_{n}(\mathcal F)(x,\infty, E) = \lim_{x'\to\infty}M_{n}(\mathcal F)(x,x',E)
	\]
	is well-defined. Then, by Theorem \ref{multilinear-christ-kiselev-lemma} and \eqref{cond1},
	\begin{equation}\label{dominantBound}
		|M_{n}(\mathcal F)(x,\infty,E)|\lqs C_0^{n}\frac{C(E)^{n}}{\sqrt{n!}}.
	\end{equation}
	Thus, the two series
	\[
	\sum_{m=1}^{\infty}M_{2m}(\mathcal{F})(x,\infty,E), \qquad 
	-\sum_{m=0}^{\infty}M_{2m+1}(\mathcal{F})(x,\infty,E)
	\]
	converge absolutely. Thus, the series in \eqref{seriesSolution2} is well-defined.
	
	On the other hand, let us see that the claimed WKB asymptotic behavior \eqref{WKBrestated2} follows from \eqref{forVerifyingWKB}: by Lebesgue dominated convergence theorem with the counting measure and the dominating sequence given in \eqref{dominantBound}, the pointwise decay \eqref{forVerifyingWKB} implies decay of the series:
	\begin{equation}\label{seriesIsZeroAtInfty}
		\lim_{x\to\infty}\sum_{m=1}^{\infty}M_{2m}(\mathcal{F})(x,\infty,E) = \sum_{m=1}^{\infty}\lim_{x\to\infty}M_{2m}(\mathcal{F})(x,\infty,E) = 0, 
	\end{equation}
and similarly for the other series, so $Y(x) \to \binom 10$ as $x \to \infty$.
	
	Finally, we verify that the series in \eqref{seriesSolution2} gives an actual solution to the differential system in \eqref{systemInY}. In order to show that 
	\[ \frac{\dd }{\dd x}\left(\sum_{m=1}^{\infty}M_{2m}(\mathcal{F})(x,\infty,E)\right) = -\mathcal F(x,E)\cdot \sum_{m=0}^{\infty} M_{2m+1}(\mathcal F)(x,\infty, E), \]
	it suffices to check for $0\lqs x<y$,
	\begin{equation}\label{DETOintegralCondition}
		\left(\sum_{m=1}^{\infty}M_{2m}(\mathcal{F})(t,\infty,E)\right)\Bigg|_{x}^{y} = \int_x^y\left(\sum_{m=0}^{\infty}-\mathcal F(t,E)M_{2m+1}(\mathcal F)(t,\infty,E)\right)\,\dd t.
	\end{equation}
	By Theorem \ref{thm2.5Liu}, $M_{2m}(\mathcal F)(x,\infty,E)\in \AC_{\loc}(\mathbb R_+)$ for any $m\in\mathbb N$, and thus
	\[ \big( M_{2m}(\mathcal{F})(x,\infty,E)\big)\Big|_x^y = \int_x^y-\mathcal F(t,E)M_{2m-1}(t,\infty,E)\,\dd t. \]
	Since the series
	\[ 
	\sum_{m=1}^{\infty}M_{2m}(\mathcal{F})(x,\infty,E) 
	\] 
	is absolutely convergent for $E\in\mathcal O$, it follows that
	\[ \left(\sum_{m=1}^{\infty}M_{2m}(\mathcal{F})(t,\infty,E)\right)\Bigg|_{x}^{y} = \sum_{m=1}^{\infty}\big(M_{2m}(\mathcal{F})(t,\infty,E)\big)\Big|_x^y; \]
	on the other hand,
	\begin{align*}
		\sum_{m=1}^{\infty}\int_x^y-\mathcal F(t,E)M_{2m-1}(t,\infty,E)\,\dd t & = \int_x^y \left(\sum_{m=1}^{\infty}-\mathcal F(t,E)M_{2m-1}(t,\infty,E)\right)\,\dd t \\
		& = \int_x^y\left(\sum_{m=0}^{\infty}-\mathcal F(t,E)M_{2m+1}(\mathcal F)(t,\infty,E)\right)\,\dd t,
	\end{align*}
	where the first equality follows from Lebesgue Dominated Convergence with the dominating function \eqref{dominantBound} and the fact that $\mathcal F\in L^1_{\loc}(\mathbb R_+)$.
	
	Similar reasoning can be applied to the odd summation $\{M_{2m+1}\}$ to conclude that
	\[ \frac{\dd }{\dd x}\left(-\sum_{m=0}^{\infty}M_{2m+1}(\mathcal{F})(x,\infty,E)\right) = \ol{\mathcal F(x,E)}  + \ol{\mathcal F(x,E)} \cdot\sum_{m=1}^{\infty}M_{2m}(\mathcal F)(x,\infty, E), \]
	where we also used the fact that 
	\[ \frac{\dd}{\dd x}M_1(\mathcal F)(x,\infty,E) = \frac{\dd}{\dd x}\int_x^{\infty} \ol{\mathcal F(t,E)}\,\dd t = -\ol{\mathcal F(x,E)}. \]
	Thus, 
	\[ Y'(x) = \begin{pmatrix}
					-\mathcal F(x,E)\cdot \sum_{m=0}^{\infty} M_{2m+1}(\mathcal F)(x,\infty, E)\\
					\ol{\mathcal F(x,E)}  + \ol{\mathcal F(x,E)} \cdot\sum_{m=1}^{\infty}M_{2m}(\mathcal F)(x,\infty, E)
				\end{pmatrix}. \]
	On the other side of the system in \eqref{systemInY}, direct computation gives
	\begin{align*}
		D(x,E)Y(x) & = \begin{pmatrix}
				0 & \cF(x,E) \\
				\ol{\cF(x,E)} & 0
				\end{pmatrix} \left(\begin{pmatrix}
		1 \\ 0
	\end{pmatrix} + \begin{pmatrix}
		\sum_{m=1}^{\infty}M_{2m}(\mathcal{F})(x,\infty,E) \\
		-\sum_{m=0}^{\infty}M_{2m+1}(\mathcal{F})(x,\infty,E)
	\end{pmatrix}\right) \\
				& = \begin{pmatrix}
					0 \\ \ol{\cF(x,E)}
				\end{pmatrix} + 
				\begin{pmatrix}
		-\mathcal F(x,E)\sum_{m=0}^{\infty}M_{2m+1}(\mathcal{F})(x,\infty,E) \\
		\ol{\cF(x,E)}\sum_{m=1}^{\infty}M_{2m}(\mathcal{F})(x,\infty,E)
	\end{pmatrix}.
	\end{align*}
	Thus, the series in \eqref{seriesSolution2} indeed gives a solution to the system \eqref{systemInY} which satisfies the WKB asymptotic in \eqref{WKBrestated2}. Thus, Theorem \ref{equivSolutions} applies at this energy.
\end{proof}

%%%%%%%%%%%%%%%%%%%%%%%%%%%%%%%%%%%%%%%%
%%%%%%%%%%%%%%%%%%%%%%%%%%%%%%%%%%%%%%%%
\section{Proof of Theorem \ref{thmWKB}}

\begin{lemma}\label{coro4.3}
Assume $\sigma, \tau - \sigma^2 \in \ell^p(L^1)$ 
	and fix a martingale structure 
	\begin{equation}
		\{E^m_j\subset\bbR_+:m\in\bbZ_+, 1\lqs j\lqs 2^m\}\quad \text{(adapted in $\ell^p(L^1)$) to}\quad |\tau - \sigma^2|+|\sigma|.
	\end{equation} 
Then, for Lebesgue-\emph{a.e.}\ $E\in \bbR_+$
	\begin{equation}\label{coro4.3eq2}
		\limsup_{M\to \infty} \|\mathcal{F}(\cdot ,E)\chi_{[M,\infty)}\|_{\mathcal{B}_1} = 0.
	\end{equation}
\end{lemma}

\begin{proof}
For readability, in this proof we write $\cB = \cB_1$. We fix a compact $K \subset \bbR_+$ and prove that \eqref{coro4.3eq2} holds for Lebesgue-a.e.\ $E \in K$. We  estimate
	\begin{align*}
		\|\mathcal{F}(\cdot, E)\chi_{[M,\infty)}\|_{\mathcal{B}} & = \left\| \left\{ \int_{E_j^m}\mathcal{F}(x,E)\chi_{[M,\infty)}(x)\,\dd x \right\} \right\|_{\mathcal{B}} \\
		& = \left\| \left\{ \int_{\bbR_+}w(x,E)e^{-ih(x,E)}\tilde{Q}(x,E)\chi_{[M,\infty)}(x)\chi^m_j(x)\,\dd x \right\} \right\|_{\mathcal{B}} \\
		& = \|\{S_w(Q\chi_{[M,\infty)}\chi^m_j)(E) + S_1(\sigma\chi_{[M,\infty)}\chi^m_j)(E)\}\|_{\mathcal{B}} \\
		& = G_{S_w(Q\chi_{[M,\infty)}), S_1(\sigma\chi_{[M,\infty)})}(E),
	\end{align*}
where $S_w, S_1$ refers to operators in \eqref{def-SandS*} with $\zeta = w$ and $\zeta =1$, respectively.

By Lemma \ref{liu-lemma2.8}, the operators $S_w,S_1$ are bounded since $K\subseteq \mathbb R_+$ is compact. Thus, by Lemma \ref{ck-JFA-lemma3.4-maximalOp},
\[
\|G_{S_w(Q\chi_{[M,\infty)}),S_1(\sigma\chi_{[M,\infty)})}(E)\|_{L^{p'}(K,\dd E)} \lqs C(\|Q\chi_{[M,\infty)}\|_{\ell^p(L^1)} + \|\sigma\chi_{[M,\infty)}\|_{\ell^p(L^1)}),
\]
where $C<\infty$ depends only on $p$ and the operator norm of $S_w, S_1$. Then, 
\begin{align*}
	\limsup_{M\to\infty}\left( \int_K\|\mathcal{F}(\cdot,E)\chi_{[M,\infty)}\|^{p'}_{\mathcal{B}}\,\dd E\right)^{1/p'} & = \limsup_{M\to\infty}\|G_{S_w(Q\chi_{[M,\infty)}), S_1(\sigma\chi_{[M,\infty)})}\|_{L^{p'}(K,\dd E)} \\
	& \lqs C(K)\cdot\limsup_{M\to\infty}(\|Q\chi_{[M,\infty)}\|_{\ell^p(L^1)} + \|\sigma\chi_{[M,\infty)}\|_{\ell^p(L^1)})\\
	& = 0,
\end{align*}
which implies that \eqref{coro4.3eq2} holds for almost every $E\in K$.
\end{proof}

\begin{proof}[Proof of Theorem \ref{thmWKB}]
If $V \in \ell^p(H^{-1})$, by Lemma~\ref{weightedEllpH-1space}	with $\gamma = 0$, there is a decomposition $V=\sigma' + \tau$ for some $\sigma, \tau$ such that $\sigma, \tau - \sigma^2 \in \ell^p(L^1)$. The proof is completed by Lemma \ref{coro4.3} and Lemma~\ref{lemmasummaryOfProof}.
\end{proof}

%%%%%%%%%%%%%%%%%%%%%%%%%%%%%%%%%%%%%%%%
%%%%%%%%%%%%%%%%%%%%%%%%%%%%%%%%%%%%%%%%
\section{Proof of Theorem \ref{thmWKBhausdorff}}

We denote by $\mathcal H^{\beta}$ the $\beta$-dimensional Hausdorff measure on $\bbR$.

\begin{lemma}\label{ck-CMP-sec8}
Let $p\in (1,2]$, $\gamma > 0$ with $\gamma p' \le 1$, where $1/p + 1/p'=1$. Assume that
\[
(1+x)^\gamma \sigma(x), (1+x)^\gamma (\tau(x) - \sigma(x)^2 ) \in \ell^p(L^1).
\]
Fix a martingale structure $\{E^m_j\subset\bbR_+:m\in\bbZ_+, 1\lqs j\lqs 2^m\}$ adapted in $\ell^p(L^1)$ to the function $(1+x)^{\gamma}(|\tau - \sigma^2|+|\sigma|)$. Denote
\[
\Lambda_c = \{E\in \bbR_+: \|\mathcal F(\cdot, E)\chi_{[N,\infty)}\|_{\cB_2}\gqs c\ \forall N  \}.
\]
Then $\mathcal H^{\beta}(\Lambda_c) = 0$ for every $\beta >1-\gamma p'$.
\end{lemma}

\begin{proof}
For readability, in this proof we write $\cB = \cB_2$ and $\mathcal{G}=\mathcal{G}^{(2)}$. Note that $\gamma \in (0,1)$ and define, for $z\in \bbC$, $\mathcal F_z(x,E) := (1+x)^z\mathcal F(x,E)$. Following the lines of argument in \cite[Section 8]{CAK2001CMP}, it suffices to fix compacts $K, J$ such that $K \subset J \subset \bbR_+$ and check that for $\Re z = \gamma$, 
	\begin{equation}\label{ck-cond1}
		\|\mathcal  F_z(\cdot, E)\|_{\mathcal B}\in L^{p'}(J,\dd E),
	\end{equation}
	and for $\Re z = \gamma - 1$,
	\begin{equation}\label{ck-cond2}
		\|\partial_E\mathcal F_z(\cdot, E)\|_{\mathcal B} \in L^{p'}(K,\dd E).
	\end{equation}
	
	\emph{Proof of \eqref{ck-cond1}}: we compute that
	\begin{align*}
 \|\mathcal F_z(\cdot, E)\|_{\mathcal B} & = \left\|\left\{ \int_{E^m_j}w(x,E)e^{-ih(x,E)}(1+x)^z\tilde{Q}(x,E)\,\dd x\right\} \right\|_{\mathcal B} \\
		& = \|\{ S_w((1+x)^zQ\chi^m_j)(E) + S_1((1+x)^z\sigma\chi^m_j)(E)\}\|_{\mathcal B} \\
		& = G_{S_w((1+x)^zQ),\ S_1((1+x)^z\sigma)}(E),
	\end{align*}
	where, as in \emph{Proof to Condition \eqref{coro4.3eq2}}, the operators $S_w, S_1$ are bounded by Lemma \ref{liu-lemma2.8}.
	
	So, by Lemma \ref{ck-JFA-proposition3.3}, 
	\begin{align*}
		\|G_{S_w((1+x)^zQ),\ S_1((1+x)^z\sigma)}(E)\|_{L^{p'}(J,\dd E)} & \lqs C\Big( \|(1+x)^zQ\|_{\ell^p(L^1)} + \|(1+x)^z\sigma\|_{\ell^p(L^1)} \Big) \\
		& = C \Big( \|(1+x)^{\gamma}Q\|_{\ell^p(L^1)} + \|(1+x)^{\gamma}\sigma\|_{\ell^p(L^1)} \Big)
	\end{align*}
	where $C<\infty$ depends only on $p,p'$ and the operator norms of $S_w, S_1$, through which it depends on the interval $J$. Note that we use $\lvert (1+x)^z \rvert = \lvert (1+x)^\gamma \rvert$ for $\Re z = \gamma$ in the last step. By assumption, $(1+x)^{\gamma}Q, (1+x)^{\gamma}\sigma\in \ell^p(L^1)$.
	
	\emph{Proof of \eqref{ck-cond2}}: note that 
	\[
	\partial_E \mathcal F_z(\cdot, E) = \frac{\partial}{\partial E} \left( w(x,E)e^{-ih(x,E)}(1+x)^z\tilde{Q}(x,E)\right)
	\]
	and the product rule will produce three terms, so we denote
	\begin{align*}
A^m_j	& = \int_{E^m_j}\frac{\partial}{\partial E} \left( w(x,E)\right)e^{-ih(x,E)}(1+x)^z\tilde{Q}(x,E)\,\dd x  \\
B^m_j	&  = \int_{E^m_j} w(x,E)\frac{\partial}{\partial E}\left(e^{-ih(x,E)}\right)(1+x)^z\tilde{Q}(x,E)\,\dd x  \\
C^m_j	& =  \int_{E^m_j} w(x,E)e^{-ih(x,E)}(1+x)^z\frac{\partial}{\partial E}\left(\tilde{Q}(x,E)\right)\,\dd x
	\end{align*}
	Since $\|\cdot\|_{\mathcal B}$ is a semi-norm by Lemma \ref{B-seminorm},
	\begin{equation}\label{Amj-Bmj-Cmj}
		\|\partial_E \mathcal F_z(\cdot, E)\|_{\mathcal B}  \lqs \|A^m_j\|_{\mathcal B} + \|B^m_j\|_{\mathcal B} + \|C^m_j\|_{\mathcal B}.
	\end{equation}
	So, in order to show that $\|\partial_E \mathcal F_z(\cdot, E)\|_{\mathcal B}\in L^{p'}(K,\dd E)$, it suffices to show that 
	\[
	\|A^m_j\|_{\mathcal B}, \|B^m_j\|_{\mathcal B}, \|C^m_j\|_{\mathcal B}\in L^{p'}(K,\dd E).
	\]
	Firstly, let us consider $\{A^m_j\}$, where $\partial_E$ lands on $w(x,E)$. We compute that
	\[
		\frac{\partial}{\partial E} \left( w(x,E)\right) = -\frac{\partial}{\partial E}\left(\frac{i}{2k}\right) = \frac{i}{4k^3}.
	\]
	Then, using the operators $S_f$ with $f=i / (4k^3)$ and $S_g$ with $g = - 1/(2k^2)$,
	\begin{align*}
		\|A^m_j\|_{\mathcal B} & = \left\|\left\{ \int_{E^m_j} \frac{i}{4k^3} e^{-ih(x,E)}(1+x)^z\tilde{Q}(x,E)\,\dd x \right\} \right\|_{\mathcal B} \\
		& = \bigg\|\bigg\{  S_f\left( (1+x)^zQ\chi^m_j \right) (E) + S_g \left( (1+x)^z\sigma\chi^m_j \right) (E)\bigg\}\bigg\|_{\mathcal B} \\
		& = G_{S_f((1+x)^zQ),\ S_g((1+x)^z\sigma)}(E),
	\end{align*}
By Lemma \ref{liu-lemma2.8}, $S_f$, $S_g$ are bounded, and by Lemma \ref{ck-JFA-proposition3.3},
	\begin{align*}
		\|G_{S_f((1+x)^zQ),\ S_g((1+x)^z\sigma)}(E)\|_{L^{p'}(K,\dd E)} &  \lqs C\Big( \|(1+x)^{\gamma-1}Q\|_{\ell^p(L^1)} + \|(1+x)^{\gamma-1}\sigma\|_{\ell^p(L^1)} \Big) \\
		& \stackrel{(*)}{\lqs} C\Big( \|(1+x)^{\gamma}Q\|_{\ell^p(L^1)} + \|(1+x)^{\gamma}\sigma\|_{\ell^p(L^1)} \Big)
	\end{align*}
	where $C=C(p,p',\|S_f\|, \|S_g\|)<\infty$ and $(*)$ holds since $(1+x)\gqs 1$. So, $\|A^m_j\|_{\mathcal B}\in L^{p'}(K,\dd E)$.
	
	The same argument works for $\{C^m_j\}$, where $\partial_E$ lands on the potential $\tilde{Q}(x,E)$. Note that
	\[ \frac{\partial}{\partial E}\left( \tilde{Q}(x,E) \right) = 2i\sigma\cdot\frac{\partial}{\partial E}\left( k \right) = \frac{i}{k}\cdot\sigma. \]
	Then,
	\begin{align*}
		\|C^m_j\|_{\mathcal B} & = \left\|\left\{ \int_{E^m_j} w(x,E)e^{-ih(x,E)}(1+x)^z\frac{\partial}{\partial E}\left(\tilde{Q}(x,E)\right)\,\dd x \right\}\right\|_{\mathcal B} \\
		& = \left\|\left\{  S_u\big( (1+x)^z\sigma\big)(E) \right\}\right\|_{\mathcal B} \\
		& = G_{S_u((1+x)^z\sigma)}(E),
	\end{align*}
	where $u = \frac ik w$ and $S_u$ is bounded by Lemma \ref{liu-lemma2.8}. By Lemma \ref{ck-JFA-proposition3.3},
	\[
	\|G_{S_u((1+x)^z\sigma)}(E)\|_{L^{p'}(K,\dd E)}\lqs C\|(1+x)^{\gamma-1}\sigma\|_{\ell^p(L^1)} \lqs C\|(1+x)^{\gamma}\sigma\|_{\ell^p(L^1)},
	\]
	where $C=C(p,p',\|S_u\|)<\infty$. Thus, $\|C^m_j\|_{\mathcal B}\in L^{p'}(K,\dd E)$.
	
Finally, we consider $\{B^m_j\}$, where $\partial_E$ lands on $e^{-ih(x,E)}$. Since 
	\[ 
	\partial_E e^{-ih(x,E)} = e^{-ih(x,E)}\cdot \big(-i \partial_E h(x,E)\big),
	\]
	where
	\begin{align*}
		\partial_E h(x,E) = \frac{\partial}{\partial E} \left(2k(E)x - \int_0^x \frac{Q(t)}{k(E)}\,\dd t \right) = \frac{x}{k(E)} + \frac{1}{2k(E)^3}\int_0^xQ(t)\,\dd t,
	\end{align*}
	it follows that
	\begin{align*}
 B^m_j = \int_{E^m_j}w(x,E)e^{-ih(x,E)} \left(\frac{\partial}{\partial E}(h(x,E)) (1+x)^z\tilde{Q}(x,E)\right)\dd x  & = \eta^m_j + \kappa^m_j + \xi^m_j + \psi^m_j,
	\end{align*}
where, by abusing the notation,
\begin{align*}
\eta^m_j & = S_{\eta}\big( x(1+x)^zQ\chi^m_j \big), \qquad \eta = \frac{1}{k(E)}\cdot w(x,E) \\
\kappa^m_j & =  S_\kappa\big( x(1+x)^z\sigma\chi^m_j \big),\qquad \kappa = 2i\cdot w(x,E)  \\
\xi^m_j & = S_\xi\left(\left(\int_0^xQ(t)\,\dd t\right) (1+x)^zQ\chi^m_j \right),\qquad \xi = \frac{1}{2k(E)^3}\cdot w(x,E) \\
\psi^m_j & = S_\psi\left(\left(\int_0^xQ(t)\,\dd t\right) (1+x)^z\sigma\chi^m_j \right),\qquad \psi = \frac{i}{k(E)^2}\cdot w(x,E)
	\end{align*}
	Consider
	\begin{align*}
		\|B^m_j\|_{\mathcal B} & = \sum_{m=1}^{\infty}m^2 \left(\sum_{j=1}^{2^m}|B^m_j|^2\right)^{1/2}.
	\end{align*}
	Note that
	\begin{equation}\label{b^m_j}
		\left\| \sum_{m=1}^{\infty}m^2 \left(\sum_{j=1}^{2^m}|B^m_j|^2\right)^{1/2}\right\|_{L^{p'}(K,\dd E)} \lqs \sum_{m=1}^{\infty}m^2 \left\|t_m(E)\right\|_{L^{p'}(K,\dd E)},
	\end{equation}
	where
	\begin{align*}
		\left\|t_m(E)\right\|_{L^{p'}(K,\dd E)}^{p'} & = \int_K\left(\sum_{j=1}^{2^m}\left| \int_{E^m_j}w(x,E)e^{-ih(x,E)}\left(\frac{\partial}{\partial E}(h(x,E)) (1+x)^z\tilde{Q}(x,E)\right)\dd x \right|^2\right)^{q/2}\dd E \\
		& \stackrel{(*)}{\lqs} 2^{m(p'/2-1)}\int_K\sum_{j=1}^{2^m}\left| \int_{E^m_j}w(x,E)e^{-ih(x,E)}\left(\frac{\partial}{\partial E}(h(x,E)) (1+x)^z\tilde{Q}(x,E)\right)\dd x \right|^{p'}\dd E \\
		& = 2^{m(p'/2-1)}\sum_{j=1}^{2^m}\int_K\left| \int_{E^m_j}w(x,E)e^{-ih(x,E)}\left(\frac{\partial}{\partial E}(h(x,E)) (1+x)^z\tilde{Q}(x,E)\right)\dd x \right|^{p'}\dd E 
	\end{align*}
	and in the step $(*)$ we used \cite[Equation (16)]{liu_multilinear}. Since
	\begin{align*}
		& \int_K \Bigg| \int_{E^m_j}w(x,E)e^{-ih(x,E)}\left(\frac{\partial}{\partial E}(h(x,E)) (1+x)^z\tilde{Q}(x,E)\right)\dd x \Bigg|^{p'}\dd E  \\
		=\ & \int_K |B^m_j|^{p'}\dd E\ \lqs\ \|\eta^m_j\|_{L^{p'}(K,\dd E)}^{p'} + \|\kappa^m_j\|_{L^{p'}(K,\dd E)}^{p'} + \|\xi^m_j\|_{L^{p'}(K,\dd E)}^{p'} + \|\psi^m_j\|_{L^{p'}(K,\dd E)}^{p'},
	\end{align*}
	it suffices to check that each of the four terms above are finite. Indeed, consider
	\begin{align*}
		\|\eta^m_j\|_{L^{p'}(K,\dd E)}^{p'} & = \left\|S_\eta\big( x(1+x)^zQ\chi^m_j \big)\right\|_{L^{p'}(K,\dd E)}^{p'} \\
		& \stackrel{(\ref{opS-bound})}{\lqs}  C  \|x(1+x)^zQ\chi^m_j\|_{\ell^p(L^1)}^{p'} \\
		& = C \|(1+x)^{z+1}Q\chi^m_j - (1+x)^zQ\chi^m_j\|_{\ell^p(L^1)}^{p'} \\
		& \lqs C \left(\|(1+x)^{z+1}Q\chi^m_j\|_{\ell^p(L^1)} + \|(1+x)^zQ\chi^m_j\|_{\ell^p(L^1)}\right)^{p'}\\
		& = C \left(\|(1+x)^{\gamma}Q\chi^m_j\|_{\ell^p(L^1)} + \|(1+x)^{\gamma-1}Q\chi^m_j\|_{\ell^p(L^1)}\right)^{p'} \\
		& \lqs C \|(1+x)^{\gamma}Q\chi^m_j\|_{\ell^p(L^1)}^{p'}.
	\end{align*}
	Thus,
	\[ \|\eta^m_j\|_{L^{p'}(K,\dd E)}^{p'}  \lqs C\cdot \|(1+x)^{\gamma}Q\chi^m_j\|_{\ell^p(L^1)}^{p'}, \]
	where $C=C(\|S_\eta\|)<\infty$ varies from line to line, but it only depends on  the operator norm $\|S_\eta\|$ in the sense of \eqref{opS-bound}, and through that, depends on the interval $K$.

	Similarly, 
	\[ 
	\|\kappa^m_j\|_{L^{p'}(K,\dd E)}^{p'} = \left\|S_\kappa\big( x(1+x)^z\sigma\chi^m_j \big)\right\|_{L^{p'}(K,\dd E)}^{p'} \lqs C\cdot \|(1+x)^{\gamma}\sigma\chi^m_j\|_{\ell^p(L^1)}^{p'}, 
	\]
	where $C=C(\|S_\kappa\|)<\infty$.
	
	Next, 
	\[ 
	\|\xi^m_j\|_{L^{p'}(K,\dd E)}^{p'} = \left\|S_\xi\left(\left(\int_0^xQ(t)\,\dd t\right) (1+x)^zQ\chi^m_j \right)\right\|_{L^{p'}(K,\dd E)}^{p'}, 
	\]
	where by Lemma \ref{weightedEllpH-1space}, $Q=\tau-\sigma^2\in \ell^p(L^1)$; thus,
	\[ \lim_{j\to\infty} \int_j^{j+1}Q(t)\,\dd t =  0 \quad\implies\quad \int_0^xQ(t)\,\dd t = O(x), \quad x \to \infty \]
	and so there exists $C<\infty$ independent of $x$ such that 
	\[
	\left\lvert \int_0^x Q(t) \, \dd t \right\rvert \le C(1+x)
	\]
and therefore
	\begin{align*}
		\left|\left(\int_0^xQ(t)\,\dd t\right)(1+x)^zQ \chi^m_j \right| 
		& \lqs C \left|(1+x)^{z+1} Q\chi^m_j\right| \\
		& = C \left|(1+x)^{\gamma}Q\chi^m_j\right|
	\end{align*}
	for any $x>0$. This pointwise inequality implies inequality of $\ell^p(L^1)$ norms,
		\begin{align*}
		\left\| \left(\int_0^xQ(t)\,\dd t\right)(1+x)^zQ\chi^m_j \right\|_{\ell^p(L^1)}  \lqs C\cdot \left\|(1+x)^{\gamma}Q\chi^m_j\right\|_{\ell^p(L^1)},
	\end{align*}
	and so we have 
	\begin{align*}
		\|\xi^m_j\|_{L^{p'}(K,\dd E)}^{p'} & = \left\|S_\xi\left(\left(\int_0^xQ(t)\,\dd t\right) (1+x)^zQ\chi^m_j \right)\right\|_{L^{p'}(K,\dd E)}^{p'} \\
		& \lqs C\cdot \left\| \left(\int_0^xQ(t)\,\dd t\right)(1+x)^zQ\chi^m_j \right\|^{p'}_{\ell^p(L^1)} \\
		& \lqs C\cdot \left\|(1+x)^{\gamma}Q\chi^m_j\right\|^{p'}_{\ell^p(L^1)}
	\end{align*}
	where $C=C(\|S_\xi\|)<\infty$ varies from line to line, but it only depends on the operator norm $\|S_\xi\|$ in the sense of \eqref{opS-bound}, and through that, depends on the interval $K$.
	
	Similarly, 
	\begin{align*}
		\|\psi^m_j\|_{L^{p'}(K,\dd E)}^{p'} & = \left\|S_\psi\left(\left(\int_0^xQ(t)\,\dd t\right) (1+x)^z\sigma\chi^m_j \right)\right\|_{L^{p'}(K,\dd E)}^{p'} \\
		& \lqs C\cdot \left\|(1+x)^{\gamma}\sigma\chi^m_j\right\|^{p'}_{\ell^p(L^1)},
	\end{align*}
	where $C=C( \|S_\psi\|)<\infty$.
	
	Thus, we have
	\begin{align*}
		\left\|t_m(E)\right\|_{L^{p'}(K,\dd E)}^{p'} & \lqs 2^{m(p'/2-1)}\sum_{j=1}^{2^m}\int_K\left| \int_{E^m_j}w(x,E)e^{-ih(x,E)}\left(\frac{\partial}{\partial E}(h(x,E)) (1+x)^z\tilde{Q}(x,E)\right)\dd x \right|^{p'}\dd E  \\
		& \lqs 2^{m(p'/2-1)}\sum_{j=1}^{2^m}\Big(\|\eta^m_j\|_{L^{p'}(K,\dd E)}^{p'} + \|\kappa^m_j\|_{L^{p'}(K,\dd E)}^{p'} + \\
		& \qquad\qquad\qquad\qquad\qquad\qquad\qquad\quad\ \|\xi^m_j\|_{L^{p'}(K,\dd E)}^{p'} + \|\psi^m_j\|_{L^{p'}(K,\dd E)}^{p'}\Big) \\
		& \lqs C\cdot 2^{m(p'/2-1)}\sum_{j=1}^{2^m}\left(\left\|(1+x)^{\gamma}Q\chi^m_j\right\|^{p'}_{\ell^p(L^1)} + \left\|(1+x)^{\gamma}\sigma\chi^m_j\right\|^{p'}_{\ell^p(L^1)}\right),
	\end{align*}
	where the constant $C=C(\|S_\eta\|, \|S_\kappa\|, \|S_\xi\|, \|S_\psi\|)<\infty$.
	
	Recall that we fixed a martingale structure $\{E^m_j\}$ adapted to 
	\[ 
	(1+x)^{\gamma}(|Q|+|\sigma|) = |(1+x)^{\gamma}Q| + |(1+x)^{\gamma}\sigma|.
	\] 
Following the lines of the proof in \cite[Proposition 3.3]{ChristKiselev}, we continue as
	\begin{align*}
		\left\|t_m(E)\right\|_{L^{p'}(K,\dd E)}^{p'} & \lqs C\cdot 2^{m(p'/2-1)}\sum_{j=1}^{2^m}\left(\left\|(1+x)^{\gamma}Q\chi^m_j\right\|^{p'}_{\ell^p(L^1)} + \left\|(1+x)^{\gamma}\sigma\chi^m_j\right\|^{p'}_{\ell^p(L^1)}\right) \\
		& \lqs C\cdot 2^{-mp'(1/p - 1/2)} \left\|(1+x)^{\gamma}(|Q|+|\sigma|)\right\|_{\ell^p(L^1)}^{p'},
	\end{align*}
	which, when plug into the original step \eqref{b^m_j}, implies that
	\begin{align*}
		\left\| \sum_{m=1}^{\infty}m^2 \left(\sum_{j=1}^{2^m}|B^m_j|^2\right)^{1/2}\right\|_{L^{p'}(K,\dd E)} & \lqs \sum_{m=1}^{\infty}m^2 \left\|t_m(E)\right\|_{L^{p'}(K,\dd E)} \\
		& \lqs C\cdot \left\|(1+x)^{\gamma}(|Q|+|\sigma|)\right\|_{\ell^p(L^1)}\sum_{m=1}^{\infty}m^2 2^{-m(1/p - 1/2)} \\
		& \lqs C\cdot \left( \left\|(1+x)^{\gamma}Q\right\|_{\ell^p(L^1)} + \left\|(1+x)^{\gamma}\sigma\right\|_{\ell^p(L^1)} \right)
	\end{align*}
	where $C<\infty$ varies from line to line. Therefore, $\|B^m_j\|_{\mathcal B} \in L^{p'}(K,\dd E)$.
	
	We have just shown that $\|A^m_j\|_{\mathcal B},\|B^m_j\|_{\mathcal B},\|C^m_j\|_{\mathcal B}\in L^{p'}(K,\dd E)$. Finally, by \eqref{Amj-Bmj-Cmj}, it follows that $\|\partial_E\mathcal F_z(\cdot, E)\|_{\mathcal B}\in L^{p'}(K,\dd E)$, and thus the claim \eqref{ck-cond2} holds.
\end{proof}

\begin{proof}[Proof of Theorem \ref{thmWKBhausdorff}]
Using the decomposition $\sigma, \tau$ provided by Lemma \ref{weightedEllpH-1space} and applying Lemma \ref{ck-CMP-sec8}, the rest of the proof follows line-by-line the arguments presented in \cite[Proof of Theorem 1.2]{liu_multilinear}.
\end{proof}

%%%%%%%%%%%%%%%%%%%%%%%%%%%%%%%%%%%%%%%%%%
%%%%%%%%%%%%%%%%%%%%%%%%%%%%%%%%%%%%%%%%%%
\section{Examples of rapidly oscillating potentials}

\begin{proof}[Proof of Example \ref{WKBoscillatoryexample}]
The given potential $V$ can be expressed as $V = \sigma' + \tau$ where
\[
\sigma(x) = - \frac 1b g(x) x^{1-b} \cos(x^b), \qquad \tau(x) = \frac 1b ( g(x) x^{1-b})' \cos(x^b)
\]
(this is motivated by an integration by parts, and checked by a direct calculation). Since $g'$ is regularly varying of index $a-1$, by Karamata's theorem \cite{Karamata30} (see also \cite{BinghamGoldieTeugels}),  $g$ is regularly varying of index $a$. Then $g(x) x^{1-b}$ is regularly varying of index $c$ and 
\[
( g(x) x^{1-b})' = g'(x) x^{1-b} + (1-b) g(x) x^{-b}
\]
is regularly varying of index $c-1$. In particular, for any $\epsilon > 0$, it follows that $\sigma = o(x^{c+\epsilon})$, $\tau = o(x^{c-1+\epsilon})$  pointwise as $x \to \infty$. It immediately follows that
\[
\left( \int_j^{j+1} \sigma(t)^2\,dt \right)^{1/2} +  \int_j^{j+1} \lvert \tau(t)\rvert \,dt = o( j^{c+\epsilon}), \qquad j \to \infty.
\]
From this, (a), (b), and (d) follow immediately. For (c), note that we obtain $\dim_H(S) \le 1 - 2\gamma$ for all $\gamma \in (0,c -1/2)$, so taking the supremum over such $\gamma$ gives $\dim_H(S) \le 2 - 2c$.
\end{proof}

\begin{proof}[Proof of Example~\ref{xmplPM1}]
It is easily obtained that $\sigma(x) = \int_0^x V(t) \,dt$ obeys $\sigma(x) = O(1/x)$ as $x\to\infty$. Thus, with $\tau=0$, Theorem~\ref{thmWKB} applies with any $p > 1$, and Theorem~\ref{thmWKBhausdorff} applies with $p=2$ and any $\gamma \in (0,1/2)$, so the claims follow.
\end{proof}

%\bibliography{mybib}
\bibliographystyle{siam}

\end{document}